\documentclass[10pt,reqno]{amsart}
\usepackage{a4wide}
\usepackage{mathbbol}
\usepackage[mathscr]{eucal}
\newtheorem{prop}{Proposition}[section]
\newtheorem{thm}[prop]{Theorem}
\newtheorem{cor}[prop]{Corollary}

\theoremstyle{remark}
\newtheorem{rem}[prop]{Remark}

\newcommand{\CC}{\mathbb{C}}
\newcommand{\cst}{\mathrm{C}^*}
\newcommand{\tens}{\otimes}
\newcommand{\id}{\mathrm{id}}
\newcommand{\st}{\:\vline\:}
\newcommand{\I}{\mathbb{1}}
\newcommand{\comp}{\!\circ\!}
\newcommand{\cS}{\mathcal{S}}
\newcommand{\cP}{\mathcal{P}}
\newcommand{\cT}{\mathcal{T}}
\newcommand{\cN}{\mathcal{N}}
\newcommand{\cB}{\mathcal{B}}
\newcommand{\cG}{\mathcal{G}}
\newcommand{\bu}{\boldsymbol{u}}
\newcommand{\is}[2]{\left(#1\,\vline\,#2\right)}
\newcommand{\bohr}{\mathfrak{b}}
\renewcommand{\Bar}[1]{\overline{#1}}
\newcommand{\SqO}{\mathrm{S}_q\mathrm{O}(3)}
\DeclareMathOperator{\M}{M}
\DeclareMathOperator{\C}{C}
\DeclareMathOperator{\Mor}{Mor}
\DeclareMathOperator{\AP}{\mathbb{AP}}
\DeclareMathOperator{\qs}{\mathscr{QS}}

\begin{document}

\title{On quantum semigroup actions on finite quantum spaces}

\date{3 October 2008}

\author{Piotr Miko{\l}aj So{\l}tan}
\address{Department of Mathematical Methods in Physics,
Faculty of Physics, University of Warsaw}
\email{piotr.soltan@fuw.edu.pl}



\thanks{Research partially supported by KBN grants nos.~N201 1770 33
\& 115/E-343/SPB/6.PRUE/DIE50/2005-2008.}

\begin{abstract}
We show that a continuous action of a quantum semigroup $\mathcal{S}$ on a finite quantum space (finite dimensional $\mathrm{C}^*$-algebra) preserving a faithful state comes from a continuous action of the quantum Bohr compactification $\mathfrak{b}\mathcal{S}$ of $\mathcal{S}$. Using the classification of continuous compact quantum group actions on $M_2$ we give a complete description of all continuous quantum semigroup actions on this quantum space preserving a faithful state.
\end{abstract}

\maketitle

\section{Introduction}

A \emph{quantum space} is an object of the category dual to the category of $\cst$-algebras (\cite[Section 0]{unbo}). A morphism in the category of $\cst$-algebras from $A$ to $B$ is a non-degenerate $*$-homomorphism $\phi:A\to\M(B)$, where $\M(B)$ is the multiplier algebra of $B$. It is known that non-degeneracy of $\phi$ allows an extension of $\phi$ to a map $\M(A)\to\M(B)$ thus providing a way to compose morphisms. For any $\cst$-algebra $A$ we write $\qs(A)$ for the quantum space corresponding to $A$. Following \cite{unbo} we will write $\Mor(A,B)$ for the set of all morphisms from $A$ to $B$.

By a \emph{quantum semigroup} we mean a quantum space $\qs(A)$ endowed with additional structure described by a morphism $\Delta_A\in\Mor(A,A\tens{A})$ such that
\[
(\Delta_A\tens\id)\comp\Delta_A=(\id\tens\Delta_A)\comp\Delta_A.
\]
If $\qs(N)$ is another quantum space (i.e.~if $N$ is a $\cst$-algebra) then an \emph{action} of a quantum semigroup $\cS=(A,\Delta_A)$ on $\qs(N)$ is a morphism $\Psi\in\Mor(N,N\tens{A})$ such that
\[
(\Psi\tens\id)\comp\Psi=(\id\tens\Delta)\comp\Psi.
\]

Let $\cS=(A,\Delta_A)$ be a quantum semigroup and let $\qs(N)$ be a quantum space. An action $\Psi_{\cS}\in\Mor(N,N\tens{A})$ of $\cS$ on $\qs(N)$ is \emph{continuous} if the set
\[
\bigl\{\Psi_{\cS}(m)(\I\tens{a})\st{m\in{N},}\;a\in{A}\bigr\}
\]
is contained and linearly dense in $N\tens{A}$. This condition is also referred to as the \emph{Podle\'s condition.} It was introduced by Piotr Podle\'s in his study of compact quantum group actions (\cite{podPHD,podles}). The term ``continuity'' used here comes from \cite{vaes}.

Let $\cS=(A,\Delta_A)$ and $\cP=(B,\Delta_B)$ be quantum semigroups. Then a \emph{quantum semigroup} morphism from $\cP$ to $\cS$ is a morphism $\Lambda\in\Mor(A,B)$ such that
\[
(\Lambda\tens\Lambda)\comp\Delta_A=\Delta_B\comp\Lambda.
\]

Let $\cS=(A,\Delta_A)$ be a quantum semigroup and let $\qs(N)$ be a quantum space such that $\dim{N}<\infty$ (in the terminology of \cite{pseu,wang} $\qs(N)$ is a \emph{finite quantum space}). Let $\Psi\in\Mor(N,N\tens{A})$ be an action of $\cS$ on $\qs(N)$. Let $\omega$ be a linear functional on $N$. Then, since $\M(N\tens{A})=N\tens\M(A)$, the map
\[
(\omega\tens\id)\comp\Psi:N\to\M(A)
\]
is well defined. We say that $\omega$ is \emph{invariant} for $\Psi$ (or that $\Psi$ \emph{preserves} $\omega$) if
\[
(\omega\tens\id)\Psi(n)=\omega(n)\I
\]
for all $n\in{N}$.

We will now briefly describe the quantum Bohr compactification of a quantum semigroup introduced in \cite{qbohr}. We will freely use the terminology of the theory of compact quantum groups from \cite{cqg}. Let $\cS=(A,\Delta_A)$ be a quantum semigroup. The quantum Bohr compactification $\bohr\cS$ of $\cS$ is a compact quantum group $\bohr\cS=\bigl(\AP(\cS),\Delta_{\AP(\cS)}\bigr)$ together with a morphism $\chi_{\cS}\in\Mor\bigl(\AP(\cS),A\bigr)$ such that for any compact quantum group $\cG=(B,\Delta_B)$ and any quantum semigroup morphism $\Lambda\in\Mor(B,A)$ there exists a unique morphism of compact quantum groups $\bohr\Lambda\in\Mor\bigl(B,\AP(\cS)\bigr)$ such that
\[
\Lambda=\chi_{\cS}\comp\bohr\Lambda.
\]
The passage from $\cS$ to $\bohr\cS$ is a functor from the category of quantum semigroups to its full subcategory of compact quantum groups (\cite{qbohr}). Let us recall that $\AP(\cS)$ is defined in \cite[Proposition 2.13]{qbohr} as the closed linear span of matrix elements of the, so called, \emph{admissible representations} of $\cS$ (\cite[Definition 2.2]{qbohr}).

\section{Actions on finite quantum spaces and the Bohr compactification}

\begin{thm}\label{main}
Let $N$ be a finite dimensional $\cst$-algebra and let $\cS=(A,\Delta_A)$ be a quantum semigroup. Let $\Psi_{\cS}\in\Mor(N,N\tens{A})$ be a continuous action of $\cS$ on $\qs(N)$. Let $\omega$ be a faithful state on $N$ invariant for $\Psi_{\cS}$. Then there exists an action $\bohr\Psi_{\cS}$ of the quantum Bohr compactification $\bohr\cS$ on $\qs(N)$ such that
\begin{equation}\label{tomabyc}
(\id\tens\chi_{\cS})\comp\bohr\Psi_{\cS}=\Psi_{\cS}.
\end{equation}
Moreover $\bohr\Psi_{\cS}$ is continuous and preserves $\omega$.
\end{thm}

The state $\omega$ defines on $N$ a scalar product $\is{\cdot}{\cdot}_{\omega}$ given by
\[
\is{x}{y}_\omega=\omega(x^*y)
\]
for all $x,y\in{N}$.

\begin{proof}[Proof of Theorem \ref{main}]
Let $\{e_1,\ldots,e_n\}$ be a basis of $N$ which is orthonormal for the scalar product $\is{\cdot}{\cdot}_\omega$. Let $\bu$ be the $n\times{n}$ matrix of elements $(u_{i,j})_{i,j=1,\ldots,n}$ of $\M(A)$ defined by
\[
\Psi_{\cS}(e_j)=\sum_{i=1}^ne_i\tens{u_{i,j}}.
\]
Then $\bu$ is a finite dimensional representation of $\cS$. We have
\[
\bu^*\bu=\I
\]
because the basis $\{e_1,\ldots,e_n\}$ is orthonormal for $\is{\cdot}{\cdot}_\omega$. Now the Podle\'s condition for $\Psi_{\cS}$ means that $\bu$ can be considered as a map $A^n\to{A^n}$. As such a map it is an isometry (of the Hilbert module $A^n$) and has dense range. Thus $\bu$ is unitary. In particular $\bu$ is invertible.

Let $\Bar{\bu}$ be the matrix obtained from $\bu$ by taking adjoint of each matrix entry of (but not transposing the matrix). We will show that $\Bar{\bu}$ is invertible.

Let $\cT$ be the (invertible) matrix $(\tau_{k,l})_{k,l=1,\ldots,n}$ such that
\[
e_i^*=\sum_{p=1}^n\tau_{p,i}e_p.
\]
Then
\[
\Psi_{\cS}(e_j^*)=\sum_{i=1}^ne_i^*\tens{u_{ij}^*}=\sum_{i,p=1}^n\tau_{p,i}e_p\tens{u_{i,j}^*}.
\]
On the other hand
\[
\Psi_{\cS}(e_j^*)=\sum_{i=1}^n\tau_{i,j}\Phi(e_i)=\sum_{p,i=1}^n\tau_{i,j}e_p\tens{u_{p,i}}.
\]
Therefore for each $p$ and $j$ we have
\[
\sum_{i=1}^n\tau_{p,i}u_{i,j}^*=
\sum_{i=1}^nu_{p,i}\tau_{i,j}.
\]
This means that
\[
(\cT\tens\I)\Bar{\bu}=\bu(\cT\tens\I).
\]
Therefore $\Bar{\bu}$ is invertible if and only if $\bu$ is invertible and the latter fact is already established. Let $\bu^\top$ be the transpose matrix of $\bu$. Then $\bu^\top=(\Bar{\bu})^*$, so $\bu^\top$ also is invertible. In other words, in the terminology of \cite{qbohr}, the matrix $\bu$ is a finite dimensional admissible representation of $\cS$.

This means that the image of the map $\Psi_{\cS}$ is contained in $N\tens\AP(\cS)\subset{N\tens\M(A)}$. Let $\bohr\Psi_{\cS}$ be the map $\Psi_{\cS}$ considered as a mapping from $N$ to $N\tens\AP(\cS)$.

It is clear that $\Psi_{\cS}$ is a morphism (it is a unital map of unital $\cst$-algebras) and satisfies \eqref{tomabyc}. Also it is easy to see that $\bohr\Psi_{\cS}$ is an action of $\bohr\cS$ on $\qs(M_2)$ preserving $\omega$. To see that $\bohr\Psi_{\cS}$ is continuous note that from $\bu\bu^*=\I$ it follows that
\[
\sum_{j=1}^n\bohr\Psi_{\cS}(e_j)(\I\tens{u_{i,j}^*})=e_i\tens\I.
\]
Thus the Podle\'s condition is satisfied for $\bohr\Psi_{\cS}$.
\end{proof}

\begin{rem}\label{remerg}
Let $\cS$, $N$, $\omega$ and $\Psi_{\cS}$ be as in Theorem \ref{main}. Assume further that the action of $\cS$ on $\qs(N)$ is \emph{ergodic}, i.e.~that for any $n\in{N}$ the condition that $\Psi_{\cS}(n)=n\tens\I$ implies $n\in\CC\I$. Then the associated action $\bohr\Psi_{\cS}$ of the quantum Bohr compactification of $\cS$ on $\qs(N)$ is ergodic as well.
\end{rem}

\begin{prop}\label{wazne}
Let $N$ be a finite dimensional $\cst$-algebra and let $\cG=(B,\Delta_B)$ be a compact quantum group. Let $\Phi_{\cG}\in\Mor(N,N\tens{B})$ be a continuous action of $\cG$ on $\qs(N)$. Then there exists a faithful state on $N$ invariant for $\Phi_{\cG}$.
\end{prop}

\begin{proof}
It is noted in \cite[Remark 2 after Lemma 4]{boca} that if $\Phi_{\cG}$ is an injective and the Haar measure $h_{\cG}$ of $\cG$ is faithful then the map
\[
E:N\ni{n}\longmapsto(\id\tens{h_{\cG}})\Phi_{\cG}(n)\in{N}
\]
is faithful. Therefore if $\varphi$ is a faithful state on $N$ then $\omega=\varphi\comp{E}$ also a faithful state. Now $\omega$ is invariant for $\Phi_{\cG}$ (\cite[Section 3]{so3}).

Since $\cG$ might not have a faithful Haar measure, we consider the action of the reduced quantum group $\cG_r=(B_r,\Delta_{B_r})$ (\cite[Page 656]{pseudogr}). If $\lambda\in\Mor(B,B_r)$ is the reducing morphism, we let $\Phi_{\cG_r}=(\id\tens\lambda)\comp\Phi_{\cG}$. Then $\Phi_{\cG_r}$ is an action of $\cG_r$ on $\qs(N)$ which is continuous. Note that the Haar measure of $\cG_r$ is faithful.

By \cite[Therem 6.8]{PodMu} the continuity of $\Phi_{\cG_r}$ guarantees that there exists a subalgebra $\cN$ on $N$ such that $\Phi_{\cG_r}$ restricted to $\cN$ is a coaction of the Hopf $*$-algebra $\cB\subset{B_r}$ on $\cN$ and $\cN$ is dense in $N$. Since $N$ is finite dimensional, we have $\cN=N$ and consequently, using the counit of $\cB$, we immediately find that $\ker{\Phi_{\cG_r}}=\{0\}$.

So far we have shown that there is a faithful state $\omega$ on $N$ invariant for $\Phi_{\cG_r}$ (because the Haar measure of the reduced group is faithful). All that remains it so note that $\omega$ is also invariant for $\Phi_{\cG}$. Indeed, for $n\in{N}$ the element $\Phi_{\cG}(n)$ belongs to $N\tens\cB$. Since the reducing map $\lambda$ is the identity on $\cB$ we have $\Phi_{\cG_r}(n)=\Phi_{\cG}(n)$ for each $n$ and
\[
(\omega\tens\id)\Phi_{\cG}(n)=\omega(n)\I.
\]
\end{proof}

\begin{thm}
Let $\cS=(A,\Delta_A)$ be a quantum semigroup and let $N$ be a finite dimensional $\cst$-algebra. Let $\Psi_{\cS}\in\Mor(N,N\tens{A})$ be a continuous action of $\cS$ on $\qs(N)$. Then $\Psi_{\cS}$ preserves a faithful state if and only if there is an action $\bohr\Psi_{\cS}$ of $\bohr\cS$ on $\qs(N)$ such that
\begin{equation}\label{to}
(\id\tens\chi_{\cS})\comp\bohr\Psi_{\cS}=\Psi_{\cS}.
\end{equation}
\end{thm}

\begin{proof}
Theorem \ref{main} provides the ``only if'' and we only need to prove the ``if'' part. If $\bohr\Psi_{\cS}$ is an action of $\bohr\cS$ on $\qs(N)$ satisfying \eqref{to}, then $\bohr\Phi_{\cS}$
is continuous. Therefore, by Proposition \ref{wazne} there is a faithful state $\omega$ on $N$ invariant for $\bohr\Psi_{\cS}$. Clearly $\omega$ is also invariant for $\Psi_{\cS}$.
\end{proof}

\section{Quantum semigroup actions on $\qs(M_2)$}

We shall now specify the situation to the case $N=M_2$. We will give a complete classification of all continuous quantum semigroup actions preserving a faithful state on $\qs(M_2)$. To that end we will use Theorem \ref{main} to reduce any such action to a continuous action of a compact quantum group. Then we will simply use the classification of such actions given in \cite{so3}.

Now for $q\in]0,1]$ let $\SqO=\bigl(\C(\SqO),\Delta_{\SqO}\bigr)$ be the quantum $\mathrm{SO}(3)$ group of Podle\'s (\cite[Remark 3]{spheres}, \cite[Section 3]{podles} and \cite[Sections 3 \& 4]{so3}). For each $q$ there is an action $\Psi_q\in\Mor\bigl(M_2,M_2\tens\C(\SqO)\bigr)$ of $\SqO$ on $\qs(M_2)$ coming from the standard action of $\mathrm{S}_q\mathrm{U}(2)$ on $\qs(M_2)$ (\cite[Subsection 3.2]{so3}, see also \cite[Lemma 2.1]{izumi}).

Quantum $\mathrm{SO}(3)$ groups can be defined as universal compact quantum groups acting on $\qs(M_2)$ and preserving the Powers state
\[
\omega_q:M_2\ni\begin{bmatrix}a&b\\c&d\end{bmatrix}\longmapsto\tfrac{1}{1+q^2}(a+q^2d)
\]
(cf.~\cite{so3}). Up to conjugation by a unitary $2\times{2}$ matrix these are all the faithful states on $M_2$. This fact made it possible to give in \cite{so3} a general theorem classifying all compact quantum group actions on $M_2$ (\cite[Thorem 6.1]{so3}). Combining this result with Theorem \ref{main} yields the following:

\begin{thm}\label{drugie}
Let $\cS=(A,\Delta_A)$ be a quantum semigroup. Let $\Psi_{\cS}\in\Mor(M_2,M_2\tens{A})$ be a continuous action of $\cS$ on $\qs(M_2)$ preserving a faithful state. Then there exists a number $q\in]0,1]$, a unitary $u\in{M_2}$ and a morphism $\Gamma\in\Mor\bigl(\C(\SqO),A\bigr)$ such that
\[
\Psi_{\cS}(m)=(\id\tens\Gamma)\bigl((u\tens\I)\Psi_q(u^*mu)(u^*\tens\I)\bigr).
\]
for all $m\in{N}$. Moreover $\Gamma$ and $u$ are unique for each $q$.
\end{thm}

Theorem \ref{drugie} is just a restatement of the universal property of Podle\'s quantum $\mathrm{SO}(3)$ groups. Namely, if $\omega$ is a faithful state on $M_2$, then there exist a unique $q$ and a unique $u$ (as in the theorem) such that $\omega(uxu^*)=\omega_q(x)$ for all $x\in{M_2}$. Therefore if $\Psi_{\cS}$ is a continuous action of $\cS=(A,\Delta_A)$ on $\qs(M_2)$ preserving $\omega$ then the associated action $\bohr\Psi_{\cS}$ of $\bohr\cS$ on $\qs(M_2)$ also preserves $\omega$. One can check that the mapping
\[
M_2\ni{m}\longmapsto(u^*\tens\I)\bohr\Psi_{\cS}(umu^*)(u\tens\I)\in{M_2}\tens\AP(\cS)
\]
is a continuous action of $\bohr\cS$ on $\qs(M_2)$ preserving $\omega_q$. The universal property of the Podle\'s groups now guarantees the existence of a unique $\Gamma_0\in\Mor\bigl(\C(\SqO),\AP(\cS)\bigr)$ intertwining this action with $\Psi_q$ (cf.~\cite{so3}). The morphism $\Gamma$ in Theorem \ref{drugie} is simply the composition of $\Gamma_0$ with $\chi_{\cS}$.

By Remark \ref{remerg}, if $\Psi_{\cS}$ is ergodic, then so is $\bohr\Psi_{\cS}$. Therefore the invariant state is unique (\cite[Lemma 4]{boca}). Therefore we have:

\begin{cor}
Let $\cS=(A,\Delta_A)$ be a quantum semigroup. Let $\Psi_{\cS}\in\Mor(M_2,M_2\tens{A})$ be a continuous ergodic action of $\cS$ on $\qs(M_2)$ preserving a faithful state. Then there exists a unique number $q\in]0,1]$, a unique unitary $u\in{M_2}$ and a unique morphism $\Gamma\in\Mor\bigl(\C(\SqO),A\bigr)$ such that
\[
\Psi_{\cS}(m)=(\id\tens\Gamma)\bigl((u\tens\I)\Psi_q(u^*mu)(u^*\tens\I)\bigr)
\]
for all $m\in{N}$.
\end{cor}


\begin{thebibliography}{66}
\bibitem{boca}
{\sc F.~Boca:} Ergodic actions of compact matrix pseudogroups on $\mathrm{C}^*$-algebras. In \emph{Recent Advances in Operator algebras.} Ast\'{e}risque \textbf{232} (1995), pp.~93--109.
\bibitem{izumi}
{\sc M.~Izumi:} Non commutative Poisson boundaries and compact quantum group actions. \emph{Adv.~Math.} \textbf{169} (2002), 1--57.
\bibitem{spheres}
{\sc P.~Podle\'s:} Quantum spheres. \emph{Lett.~Math.~Phys.} \textbf{14} (1987), 193--202.
\bibitem{podPHD}
{\sc P.~Podle\'s:} Przestrzenie kwantowe i ich grupy symetrii (Quantum spaces and their symmetry groups). PhD Thesis, Department of Mathematical Methods in Physics, Faculty of Physics, Warsaw University (1989) (in Polish).
\bibitem{podles}
{\sc P.~Podle\'{s}:} Symmetries of quantum spaces. Subgroups and quotient spaces of quantum $\mathrm{SU}(2)$ and $\mathrm{SO}(3)$ groups. \emph{Commun.~Math.~Phys.} \textbf{170} (1995), 1--20.
\bibitem{PodMu}
{\sc P.~Podle\'s \& E.~M\"uller:} Introduction to quantum groups. \emph{Rev.~Math.~Phys.} \textbf{10} no.~4 (1998), 511--551.
\bibitem{qbohr}
{\sc P.M.~So\l{}tan:} Quantum Bohr compactification. \emph{Ill.~J.~Math.} \textbf{49} no.~4 (2005), 1245--1270.
\bibitem{so3}
{\sc P.M.~So\l{}tan:} Quantum $\mathrm{SO}(3)$ groups quantum group actions on $M_2$. arXiv:0810.0398v1 [math.OA]. Submitted to \emph{Journal of Noncommutative Geometry.}
\bibitem{vaes}
{\sc S.~Vaes:} A new approach to induction and imprimitivity results. \emph{J.~Funct.~Anal.} \textbf{229} (2005), 317--374.
\bibitem{wang}
{\sc S.~Wang:} Quantum symmetry groups of finite spaces. \emph{Commun.~Math.~Phys.} \textbf{195} (1998), 195--211.
\bibitem{pseu}
{\sc S.L.~Woronowicz:} Pseudogroups, pseudospaces and Pontryagin duality.
\emph{Proceedings of the International Conference on Mathematical Physics,
Lausanne 1979} Lecture Notes in Physics, \textbf{116}, pp.~407--412.
\bibitem{pseudogr}
{\sc S.L.~Woronowicz:} Compact matrix pseudogroups. \emph{Comm.~Math.~Phys.} \textbf{111} (1987), 613--665.
{\bibitem{unbo}
{\sc S.L.~Woronowicz:} Unbounded elements affiliated with $\mathrm{C}^*$-algebras and non-compact quantum groups. \emph{Commun.~Math.~Phys.} \textbf{136} (1991), 399--432.}
\bibitem{cqg}
{\sc S.L.~Woronowicz:} Compact quantum groups. In: \emph{Sym\'etries
quantiques, les Houches, Session LXIV 1995,} Elsevier 1998, pp.~845--884, Elsevier 1998.
\end{thebibliography}
\end{document}